\documentclass[a4paper,11pt]{amsart}
\usepackage{amsfonts,amsthm,amsmath,amssymb}
\usepackage{graphicx}
\usepackage{epsfig}
\usepackage{tikz}
\usepackage{hyperref}
\theoremstyle{plain}
\newtheorem{theorem}{Theorem}[section]
\newtheorem{prop}[theorem]{Proposition}

\newtheorem{corollary}[theorem]{Corollary}

\theoremstyle{remark}
\newtheorem{remark}[theorem]{Remark}

\theoremstyle{definition}
\newtheorem{definition}[theorem]{Definition}
\newtheorem{example}[theorem]{Example}

\newcommand{\CC}{{\mathbb{C}}}
\newcommand{\QQ}{{\mathbb{Q}}}
\newcommand{\ZZ}{{\mathbb{Z}}}

\newcommand{\Gr}{{\mathop{\mathrm{Gr}}}}
\newcommand{\Fl}{{\mathop{\mathrm{Fl}}}}

\newcommand{\SL}{{\mathop{\mathrm{SL}}}}

\DeclareMathOperator{\rk}{rk}
\DeclareMathOperator{\lct}{lct}
\DeclareMathOperator{\Pic}{Pic}

\def\udot {{\:\raisebox{3pt}{\text{\circle*{1.5}}}}}

\def \leq {\leqslant}
\def \geq {\geqslant}

\title[Flag varieties and Hwang's product theorem]{Singularities of divisors on flag varieties\\ via Hwang's product theorem}
\author{Evgeny Smirnov}
\email{esmirnov@hse.ru}
\address{Faculty of Mathematics and Laboratory of Algebraic Geometry and its Applications, National Research University Higher School of Economics, \mbox{Usacheva} str., 6, Moscow 119048, Russia}
\address{Independent University of Moscow, Bolshoi Vlassievskii per., 11, Moscow 119002, Russia}
\date{\today}
\thanks{The study was partially supported by the Russian Academic Excellence Project \mbox{``5--100''}, Simons--IUM fellowship (\S 1--3), and RSF grant, project 14-21-00053 dated 11.08.14 (\S4).}

\begin{document}

\begin{abstract}
We give an alternative proof of a recent result by B.~Pasquier stating that for a generalized flag variety $X=G/P$ and an effective $\QQ$-divisor $D$ stable with respect to a Borel subgroup the pair~\mbox{$(X,D)$} is Kawamata log terminal if and only if~\mbox{$\lfloor D\rfloor=0$}.
\end{abstract}

\maketitle

\section{Introduction}

Let $G$ be a connected reductive algebraic group over $\CC$. Recall that a horospherical $G$-variety $X$ is a normal $G$-variety with an open $G$-orbit isomorphic to a torus fibration~\mbox{$G/H$} over a flag variety $G/P$, where $P$ is a parabolic subgroup in $G$ and $P/H$ is a torus. In \cite{Pasquier15} Boris Pasquier shows that for a horospherical variety $X$ and an effective $\QQ$-divisor $D$ stable with respect to a
Borel subgroup the pair $(X,D)$ is Kawamata log terminal if and only if $\lfloor D\rfloor=0$.

An essential part of the proof is the case when $X$ itself is a flag variety~$G/P$.
In this case Pasquier uses a Bott--Samelson resolution to provide an explicit log resolution of the pair~\mbox{$(G/P,D)$},
and to check that this pair is Kawamata log terminal using rather heavy combinatorics related to root systems. Further on,
he uses the latter resolution to provide a log resolution of a general horospherical pair,
and he uses Kawamata log terminality of~\mbox{$(G/P,D)$} to establish the same result in general.

The main purpose of this note is to prove a similar result for a variety~\mbox{$G/P$} avoiding explicit
log resolutions, and instead using the Product Theorem for log canonical thresholds
due to Jun-Muk Hwang, see~\cite{Hwang07}. In particular, we will not assume that the $\QQ$-divisor $D$ is stable under the action of the Borel subgroup.
Note that while this approach does not allow one to get rid of Bott--Samelson resolutions because they are
needed for the case of general horospherical varieties, it does allow to avoid the computations from~\mbox{\cite[\S5]{Pasquier15}}
and to replace them by easier computations of Proposition~\ref{prop:restriction} below.

The plan of the paper is as follows.
In \S\ref{section:lct} we recall definitions
and properties of log canonical thresholds.
In \S\ref{section:notation} we recall the basic facts about geometry of
flag varieties and give a precise statement of our main result,
which is Theorem~\ref{thm:main}.
In \S\ref{section:proof} we introduce more notation and prove the main theorem.

I am grateful to I.\,Cheltsov, Yu.\,Prokhorov and especially C.~Shramov for very useful discussions and to M.~Brion and B.~Pasquier for valuable comments.

\section{Log canonical thresholds}
\label{section:lct}

In this section we recall definitions and some properties of log canonical
thresholds. We refer a reader to~\cite[\S8]{Kollar97} for (much) more details.

Let $X$ be a smooth
complex algebraic variety, $D$ be an effective $\QQ$-divisor. Choose a point $x\in X$.

\begin{definition}\label{definition:lct}
If $D$ is a Cartier divisor locally defined
by the equation~\mbox{$f=0$}, then the log canonical threshold $\lct_x(D)$ of $D$ near $x$ is defined by
\[
\lct_x(D)=\sup\left\{c>0\ \left| \  \frac{1}{|f|^{c}}\in L^2_{loc}\right.\right\}.
\]
In particular, if $x$ is not contained in the support of~$D$, we put~\mbox{$\lct_x(D)=+\infty$}.

If $D$ is an arbitrary $\QQ$-divisor, then we define $\lct_x(D)=\lct_x(rD)/r$ for sufficiently divisible
integer $r$.
The log canonical threshold $\lct(X,D)$ of $D$ is defined as the infimum of~\mbox{$\lct_x(D)$} over all~\mbox{$x\in X$}.
\end{definition}

\begin{definition} The pair $(X,D)$ is said to be \emph{Kawamata log terminal} if the inequality~\mbox{$\lct(X,D)>1$}
holds, and \emph{log canonical} if the inequality~\mbox{$\lct(X,D)\geq 1$} holds.
\end{definition}

Let $L\in\Pic(X)$ be a line bundle such that the linear system $|L|$ is non-empty.
We define $\lct(X,L)$ as the infimum $\inf_{\Delta\in |L|}\lct(X,\Delta)$.
The following theorem is taken from \cite{Hwang07}.

\begin{theorem}[{see \cite[\S2]{Hwang07}}]
\label{theorem:Hwang}
Let $f\colon X\to Y$ be a smooth projective morphism between two smooth projective varieties, $y\in Y$, and $X_y=f^{-1}(y)$ be the fiber over $y$. Let~$D$ be an effective divisor on $X$, and let $L$ be the restriction
of~\mbox{$\mathcal{O}_X(D)$} to $X_y$. Then one of the following holds:
\begin{itemize}
\item[(i)] either $\lct_x(D)\geq\lct(X_y,L)$ for each $x\in X_y$,

\item[(ii)] or $\lct_{x_1}(D)=\lct_{x_2}(D)$ for any two points $x_1,x_2\in X_y$.
\end{itemize}
\end{theorem}

One can define the global log canonical threshold
$$
\lct(X)=\inf\left\{\lct(X,\Delta)\mid \Delta\sim_{\QQ} -K_X \text{\ is an effective $\QQ$-divisor\ }  \right\},
$$
where $-K_X$ is the anticanonical class of $X$. This
definition makes sense if some positive multiple of $-K_X$ is effective;
for example, this holds for Fano varieties
and for spherical varieties (see \cite[\S4]{Brion97}).
For more properties of $\lct(X)$ see \cite{CheltsovShramov08};
for its relation to the $\alpha$-invariant of Tian see~\cite[Appendix~A]{CheltsovShramov08}.

\section{Flag varieties}
\label{section:notation}

Let $G$ be a connected reductive algebraic group. We fix a Borel subgroup~\mbox{$B$} in $G$ and a maximal torus $T\subset B$. Denote by $R$ the root system of~$G$, and by $S\subset R$ the set of simple roots in $R$, where the positive roots are the roots of $(B,T)$. The Weyl group of $R$ will be denoted by $W$; let~\mbox{$\ell\colon W\to \ZZ_{\geq 0}$} be the length function on $W$.

Let $P\supset B$ be a parabolic subgroup in $G$. Then $G/P$ is a (generalized) partial flag variety. Denote by $I$ the set of simple roots of the Levi subgroup of $P$; in particular, for $P=B$ we have $I=\varnothing$ and for  $P$ maximal  the set~$I$ is obtained from $S$ by removing exactly one simple root. For example, if~\mbox{$G=\SL_n(\CC)$}, then for the maximal parabolic subgroup $P$ corresponding to~\mbox{$S\setminus\{\alpha_k\}$, $1\leq k\leq n-1$}, the homogeneous space $G/P$ is the Grassmannian~\mbox{$\Gr(k,n)$}.

For a subset $I\subset S$ of the set of simple roots, let $W_P\subset W$ be the subgroup of  $W$ generated by the simple reflections $s_\alpha$, where $\alpha\in I$. In each left coset from $W/W_P$ there exists a unique element of minimal length. Denote the set of such elements by $W^P$; we will identify it with $W/W_P$. It is well-known
(see, for instance, \cite[\S1.2]{Brion05})
that the partial flag variety $G/P$ admits a Schubert decomposition into orbits of~$B$, and the orbits are indexed by the elements of $W^P$:
\[
G/P=\bigsqcup_{w\in W^P} BwP/P.
\]
Moreover, the dimension of the cell $BwP/P$ equals the length of $w$. The closures of these cells are called \emph{Schubert varieties}; we denote them by $Y_w=\overline{BwP/P}$.

Denote by $w_0$ and $w_0^P$ the longest elements in $W$ and $W_P$, respectively. Then the length $\ell(w_0w_0^P)$ is the dimension of $G/P$. We shall also need the \emph{opposite Schubert varieties} 
$$
Y^w=\overline{w_0Bw_0w P/P}=w_0 Y_{w_0w}=w_0 Y_{w_0w}.
$$

If $w\in W^P$, then $w_0ww_0^P$ also belongs to $W^P$ (i.e., is the shortest representative in its left coset $w_0w W_P=w_0ww_0^P W_P$); in this case $\dim Y_w=\ell(w)$ and 
$$
\dim Y^w=\ell(w_0w_0^P)-\ell(w).
$$ 
From the definition of $Y_w$ and $Y^w$ we readily see that the cohomology classes~$[Y_w]$ and $[Y^{w_0ww_0^P}]$ in $H^\udot(G/P,\ZZ)$ are equal.

Irreducible $B$-stable divisors of $G/P$ are the Schubert varieties of codimension 1. Denote them by
$$
D_\alpha=\overline{B w_0 s_\alpha w_0^P P/P}=w_0 Y^{s_\alpha}.
$$

The following proposition is a standard fact on Schubert varieties (cf. \mbox{\cite[\S1.4]{Brion05}}).
\begin{prop}\label{prop:schubert}
\begin{itemize}

\item[(i)] The divisors $D_\alpha$ for $\alpha\notin I$
freely generate~\mbox{$\Pic(G/P)$}, so one has $\rk\Pic(G/P)=|S\setminus I|$. In particular, for~$P$ maximal  one has $\Pic(G/P)\cong \ZZ$, and there is a unique $B$-stable prime divisor.

\item[(ii)] The classes of Schubert varieties $[Y_w]\in H^\udot(G/P,\ZZ)$ freely 
gene\-rate~\mbox{$H^\udot(G/P,\ZZ)$} as an abelian group. The elements of this basis are Poincar\'e dual to the classes of the corresponding opposite Schubert varieties: if $w,v\in W^P$ and $\ell(w)=\ell(v)$, then
\[
 [Y^w]\smile [Y_v]=\delta_{wv}\qquad \text{for each }w,v\in W.
\]
In particular, the classes of one-dimensional Schubert varieties, that is, of $B$-stable curves  $\overline{Bs_\alpha P/P}\subset G/P$, are dual to the classes of divisors:
\[
D_\alpha\smile [\overline{Bs_\beta P/P}]=\delta_{\alpha\beta}\qquad\text{for each }\alpha,\beta\in S\setminus I.
\]
\end{itemize}
\end{prop}

\smallskip
The purpose of our paper is to give a new proof of the following result.

\begin{theorem}[{see \cite[Theorem~3.1]{Pasquier15}}]
\label{thm:main} Let $D\sim\sum a_\alpha D_\alpha$, where $a_\alpha$ are non-negative rational numbers,
be an effective non-zero $\QQ$-divisor. Then
\begin{equation*}
\lct(G/P,D)\geq \frac{1}{\max a_{\alpha}}.
\end{equation*}
In particular,
the pair $(G/P,D)$ is Kawamata log terminal provided that all $a_\alpha$ are less than 1.
\end{theorem}

\begin{remark}
One can show that every effective divisor on $G/P$
is linearly equivalent to an \emph{effective} $B$-stable
divisor. In other words, the classes of divisors $D_{\alpha}$ in the
$\QQ$-vector space $\Pic(G/P)\otimes\QQ$ span the
cone of effective divisors. Thus the assumption of Theorem~\ref{thm:main}
requiring that $a_{\alpha}$ are non-negative is implied by effectiveness of
$D$; we keep it just to make the assertion more transparent.
\end{remark}

\begin{remark}
If the $\QQ$-divisor $D$ of Theorem~\ref{thm:main} is $B$-stable,
then in addition to the inequality given by Theorem~\ref{thm:main}
we have an obvious opposite inequality, because $D_{\alpha}$
is an effective divisor. Therefore, in this case we recover the
equality given by~\cite[Theorem~3.1]{Pasquier15}.
\end{remark}

A by-product of Theorem~\ref{thm:main} is the following assertion
on global log canonical thresholds of complete flag varieties
that is well known to experts.

\begin{corollary}
One has $\lct(G/B)=1/2$.
\end{corollary}
\begin{proof}
One has $-K_{G/B}\sim \sum_{\alpha\in S}2 D_{\alpha}$.
Thus $\lct(G/B)\geq 1/2$ by Theorem~\ref{thm:main}.
The opposite inequality is implied by the fact that the divisor
$D_{\alpha}$ is effective.
\end{proof}

\section{Proof of the main theorem}
\label{section:proof}

In this section we prove Theorem~\ref{thm:main}.

Fix a simple root $\alpha\in S\setminus I$. Let $J=I\cup\{\alpha\}$, and let $P'$ be the parabolic subgroup corresponding to $J$; then $P\subset P'$. There is a $G$-equivariant fibration~\mbox{$\pi_\alpha\colon G/P\to G/P'$}.

Let $X_\alpha$ be a fiber of this fibration. Consider the Dynkin diagram of $G$; its vertices correspond to simple roots from $S$. Let $\overline J$ be the connected component containing $\alpha$ of the subgraph spanned by the vertices of $J$. This component is the Dynkin diagram of a connected simple algebraic group $\overline G$. Let $\overline{P}_\alpha$ be a maximal parabolic subgroup of $\overline G$ with the set of roots $\overline J\setminus \{\alpha\}$. Then $X_\alpha$ is isomorphic to the $\overline G$-homogeneous space $\overline G/\overline{P}_\alpha\cong P'/P$.

\begin{example} Let $G=\SL_n(\CC)$. Its Dynkin diagram is $A_{n-1}$; denote its simple roots by $\alpha_1,\ldots,\alpha_{n-1}$. Put $I=S\setminus\{\alpha_{d_1},\ldots,\alpha_{d_r}\}$, where~\mbox{$1\leq d_1<\ldots<d_r\leq n-1$}. We also formally set $d_0=0$ and $d_{r+1}=n$.
Then $G/P$ is a partial flag variety
$$
\Fl(d_1,\ldots,d_r)\cong\{ U_1\subset\ldots\subset U_r\subset \CC^n\mid \dim U_j=d_j\}.
$$
Let $\alpha=\alpha_{d_s}$,
where $1\leq s\leq r$, and $J=I\cup \{\alpha\}$.
Then $G/P'$ is an $(r-1)$-step flag
variety $\Fl(d_1,\ldots,\widehat{d_s},\ldots,d_r)$,
and the map $\pi_\alpha\colon G/P\to G/P'$ is given by forgetting
the $s$-th component of each flag. The fibers of this
projection are isomorphic to the
Grassmannian $\Gr(d_s-d_{s-1},d_{s+1}-d_{s-1})$.
\end{example}

According to Proposition~\ref{prop:schubert} (i), since $\overline P_\alpha\subset \overline G$ is a maximal parabolic subgroup,
one has $\Pic X_\alpha\cong\ZZ$.
Let $H_\alpha$ be the ample generator of $\Pic X_\alpha$.

\begin{theorem}[{\cite[Theorem 2]{Hwang07}; see also \cite{Hwang06}}]\label{thm:hwanggr}
Let $k$ be a positive integer. Then one has~\mbox{$\lct(X_\alpha, kH_\alpha)=1/k$}.
\end{theorem}

\begin{remark}
The assertion of \cite[Theorem 2]{Hwang07} is that the inequa\-lity~\mbox{$\lct(X_\alpha, kH_\alpha)\geq 1/k$} holds.
This is equivalent to Theorem~\ref{thm:hwanggr} since the linear system $|H_{\alpha}|$ is always
non-empty by Proposition~\ref{prop:schubert}(i).
\end{remark}

The following computation will be the central point of our proof of Theorem~\ref{thm:main}.

\begin{prop}\label{prop:restriction}  For each $\beta\in S\setminus I$ one has
\[
D_\alpha\vert_{X_\beta}\sim \begin{cases} H_\alpha & \text{if }\alpha=\beta;\\
0 & \text{otherwise}.
\end{cases}
\]
\end{prop}

\begin{proof} As we discussed above, the fiber $X_\beta$ can be identified with the variety~\mbox{$P'/P\subset G/P$}. It is a flag variety with the Picard group of rank one; since the classes of $B$-stable curves are dual to the classes of ($B$-stable) divisors (see Proposition~\ref{prop:schubert} (ii)), $X_\beta$ contains a unique $B$-stable curve. This curve has the form $\overline{Bs_\beta P/P}\subset P'/P\cong X_\beta$.  Its class in $H^\udot(X_\beta,\ZZ)$ is Poincar\'e dual to the ample generator $H_\beta$ of $\Pic(X_\beta,\ZZ)$.


At the same time, as it was stated in Proposition~\ref{prop:schubert} (ii), the intersection of~$\overline{Bs_\beta P/P}$ with $D_\alpha$ equals the class of a point if $\alpha=\beta$ and zero otherwise.
\end{proof}

\begin{proof}[Proof of Theorem~\ref{thm:main}.]
Replacing $D$ by its appropriate multiple, we may assume that it is a Cartier divisor, not just
a $\QQ$-divisor. Put $a=\max a_{\alpha}$. Suppose that $\lct(X,D)<1/a$.

Pick a point $x\in X$ such that $\lct_x(D)<1/a$. Choose an index $\alpha$ from $S\setminus I$,
and let $[D|_{X_\alpha}]$ be the class of the restriction of $\mathcal{O}_X(D)$ to $X_{\alpha}$ in $\Pic(X_{\alpha})$.
According to Proposition~\ref{prop:restriction}, one has $[D|_{X_\alpha}]\sim a_\alpha H_\alpha$.  Theorem~\ref{thm:hwanggr} implies that
$$
\lct(X_\alpha, [D|_{X_\alpha}])\geq \frac{1}{a_{\alpha}}\geq \frac{1}{a};
$$
this trivially includes the case when $a_{\alpha}=0$.
Without loss of generality we can suppose that the fiber $X_\alpha$ passes through the point $x$.
The above means that the alternative (i) in Theorem~\ref{theorem:Hwang} never holds.

Let $\alpha_1,\ldots,\alpha_r$ be the simple roots from $S\setminus I$. Let $\widetilde X_1$ be the fiber of~$\pi_{\alpha_1}$ passing through the point $x$. For each $i=2,\ldots,r$ let $\widetilde X_i$ be the union of all fibers of $\pi_{\alpha_i}$ passing through the points of $\widetilde X_{i-1}$. In particular, one has~\mbox{$\widetilde X_r=X$}.

First apply Theorem~\ref{theorem:Hwang} to the point $x$ and the fibration $\pi_{\alpha_1}$. It implies that for each $x_1\in \widetilde X_1$ we have $\lct_{x_1}(D)<1/a$. Now apply it to each point $x_1\in \widetilde X_1$ and the fibration $\pi_{\alpha_2}$. We see that for each $x_2\in\widetilde X_2$ the inequality $\lct_{x_2}(D)<1/a$ holds. Proceeding by induction, we obtain the same inequality for every point in~\mbox{$\widetilde X_r=X$}. In particular, each point of $X$ is contained in the support of $D$, which is a contradiction.
\end{proof}

\def\cprime{$'$}

\end{document}